\documentclass[a4paper,12pt]{article}
\usepackage{amsmath, amsthm, amstext}
\usepackage{amssymb, array, amsfonts}

\numberwithin{equation}{section}

\def \al{\alpha}

\def \ga{\gamma}

\def \er{\varepsilon}
\def \ze{\zeta}

\def \ph{\varphi}
\def \oo{\omega}

\def \D{\Delta}

\def \O{\Omega}

\def \R{\mathbb{R}}

\def\n{\nabla}
\def\dd{\partial}
\def\div{\operatorname{div}}

\def\1{1\!\!\!\!1}

\theoremstyle{plain}
\newtheorem{theorem}{\bf Theorem}[section]
\newtheorem{lemma}[theorem]{\bf Lemma}

\theoremstyle{remark}
\newtheorem{rem}[theorem]{\bf Remark}

\renewcommand{\le}{\leqslant}
\renewcommand{\ge}{\geqslant}

\renewcommand{\qed}{\vrule height7pt width5pt depth0pt}

\title{On the regularity of solutions to the equation
$-\D u + b \cdot \n u = 0$}

\author{N. Filonov
\thanks{This work is supported by RFBR grant 11-01-00324.}
\footnote{St.Petersburg Department of Steklov Mathematical Institute,
27 Fontanka, 191023 St.Petersburg, and
St.Petersburg State University, Physics Faculty.}
}
\date{}
\begin{document}
\maketitle

\begin{abstract}
The equation $-\D u + b \cdot \n u = 0$ is considered.
The dependence of the local regularity of a solution $u$
on the properties of the coefficient $b$ is investigated.
\end{abstract}

\hfill{\it To the memory of O.~A.~Ladyzhenskaya}

\section{Formulation of the results}

Denote by $B_R$ a ball in $\R^n$, $n \ge 2$,
of radius $R$ centered at the origin.
We consider the equation
\begin{equation}
\label{11}
-\D u + b \cdot \n u = 0 
\end{equation}
in $B_R$.
We always assume that a scalar function $u \in W_2^1 (B_R)$,
and a vector-valued coefficient $b\in L_p (B_R)$, $p\ge 2$.
We understand the equation \eqref{11} in the sense of the integral identity
\begin{equation*}
\int_{B_R} \n u \cdot (\n h + b h)\, dx = 0 
\quad \forall h \in C_0^\infty (B_R) .
\end{equation*}
We are interested in the dependence of the local regularity of 
the solution $u$ of \eqref{11} on the order $p$ of the summability 
of the coefficient $b$.
The aim of the present paper is to list the results, 
and the counterexamples which guarantee the sharpness of the results.
The brief summary is given in the Table 1 below.

The critical case is $p=n$.
If $p>n$, the solution $u$ is continuously differentiable.

\begin{theorem}[\cite{LU}, Chapter III, Theorem 15.1]
\label{t1}
Let $b \in L_p (B_R)$, $p>n$, and let $u \in W_2^1(B_R)$ 
be a solution to the equation \eqref{11}.
Then 
$$
u \in W_p^2 (B_r) \subset C^{1, 1-\frac np} (B_r) \quad 
\forall r <R .
$$
\end{theorem}

Here and in what follows by $u \in W_p^2 (B_r)$ we mean that 
the restriction of $u$ onto the ball $B_r$ belongs to this space,
$\left. u \right|_{B_r} \in W_p^2 (B_r)$.

If $p=n$ the properties of solution depend on the dimension,
whether $n=2$ or $n>2$.

\subsection{Case $n=2$}
Let us consider two simple examples. 
The first example shows that when $p=n=2$ a solution $u$ can be unbounded.
The second one shows that even if we assume a priori a solution to be bounded,
then it can fail to be H\" older continuous.

{\bf Example 1.}
Let $n=2$, $R=1/e$,
$$
u(x)= \ln |\ln|x||, \quad b(x) = \frac{-x}{|x|^2 \ln |x|} .
$$
Then $b \in L_2 (B_{1/e})$, 
$u \in \mathring W_2^1 (B_{1/e})$, and \eqref{11} is satisfied,
but $u \notin L_\infty (B_{1/e})$.

{\bf Example 2.}
Let $n=2$, $R=1/2$,
$$
u(x)= \frac1{\ln|x|}, \quad 
b(x) = - \frac{2x} {|x|^2 \ln |x|} .
$$
Then $b \in L_2 (B_{1/2})$, 
$u \in W_2^1 (B_{1/2}) \cap C(\overline{B_{1/2}})$, and \eqref{11} is satisfied.
But $u \notin C^\al (B_{1/2})$ for any $\al >0$.

The situation changes if the coefficient $b$ 
satisfies an extra condition $\div b = 0$.

\begin{theorem}
\label{t12}
Let $n=2$, $b \in L_2 (B_R)$ and $\div b = 0$.
Let $u \in W_2^1(B_R)$ be a solution to equation \eqref{11}.
Then 
$$
u \in \bigcap_{q<2} W_q^2 (B_r) \subset \bigcap_{\al<1} C^\al (B_r) \quad 
\forall r <R .
$$
\end{theorem}

We prove Theorem \ref{t12} in the next section.

\begin{rem}
In \cite{NU} a more general equation
\begin{equation}
\label{12}
- \div (a \n u) + b \cdot \n u = 0 
\end{equation}
is considered.
The matrix-coefficient $a(x)$ is assumed to be positive and bounded,
\begin{equation}
\label{13}
0 < \al_0 \1 \le a(x) \le \al_1 \1 \ ,
\end{equation}
here $\1$ is the identity matrix.
If $b \in L_2 (B_R)$, $\div b = 0$, then a solution $u$ 
to \eqref{12} is H\" older continuous, $u \in C^\al$ with some $\al > 0$
(see Corollary 2.3 and the comments at the end of \S 2 in \cite{NU}).
\end{rem}

\begin{rem}
If the coefficient $b$ satisfies a slightly 
stronger condition than $b \in L_2$,
$$
\int_{B_R} |b(x)|^2 \ln (1+|b(x)|^2)\, dx < \infty
$$
(without the divergence-free condition), 
then the statement of Theorem \ref{t12} remains valid, 
see \S \ref{p44} below.
\end{rem}

\subsection{Case $n\ge 3$}
In this case, the condition $b \in L_n$ is sufficient for $u$
to be H\" older continuous.

\begin{theorem}
\label{t2}
Let $n \ge 3$, $b \in L_n (B_R)$, 
and $u \in W_2^1 (B_R)$ be a solution to equation \eqref{11}.
Then 
\begin{equation*}
u \in \bigcap_{q<n} W_q^2 (B_r) \subset \bigcap_{\al < 1} C^\al (B_r) \quad 
\forall r < R .
\end{equation*}
\end{theorem}

This theorem is probably known, 
although we have not found a relevant reference.
Theorem \ref{t2} can be proved in the same way that Theorem \ref{t12},
see Remark \ref{r26} below.

The following example shows that a solution $u$
can be unbounded when $p<n$.

{\bf Example 3.}
Let $n \ge 3$, $R=1$,
$$
u(x) = \ln |x|, \quad b(x) = \frac{(n-2)x}{|x|^2} .
$$
Then $b \in L_p (B_1)$ for all $p<n$, 
$u \in \mathring W_2^1 (B_1)$, and \eqref{11} is satisfied,
but $u \notin L_\infty (B_1)$.

Furthermore, for $p<n$, if we assume a priori a solution to be bounded,
it can be discontinuous, even for divergence-free coefficient $b \in L_p$.

\begin{theorem}
\label{t4}
Let $n \ge 3$, $p<n$.
There exist a vector-function $b_0 \in L_p (B_{1/2})$, $\div b_0 = 0$,
and a scalar function $u_0 \in W_2^1 (B_{1/2}) \cap L_\infty (B_{1/2})$ 
such that the equation \eqref{11} is satisfied, but $u_0 \not\in C(B_{1/2})$.
\end{theorem}

We prove this Theorem in Section 3.

\begin{rem}
It is easy to construct an example of a bounded solution which is not 
H\" older continuous for the case $\div b \neq 0$.
\end{rem}

{\bf Example 4.}
Let $n\ge 3$, $R=1/2$,
$$
u(x)= \frac1{\ln|x|}, \quad 
b(x) = \left( {(n-2)}{|x|} -  \frac2 {|x| \ln |x|} \right) \frac{x}{|x|} .
$$
Then $b \in \cap_{p<n} L_p (B_{1/2})$,
$u \in W_2^1 (B_{1/2}) \cap C(\overline{B_{1/2}})$, and \eqref{11} is satisfied.
But $u \notin C^\al (B_{1/2})$ for any $\al >0$.

For the proof of Theorem \ref{t4} we follow the approach 
of the paper \cite{SSSZ}.
We consider together with \eqref{11} the equation
\begin{equation}
\label{16}
- \D u + \div (b u) = 0 .
\end{equation}
We understand this equation in the sense
\begin{equation*}
\int_{B_R} u (\D h + b \cdot \n h)\, dx = 0 \ \ \forall h \in C_0^\infty (B_R) ;
\end{equation*}
the integral is well defined if $u, b \in L_2 (B_R)$.
It is clear that every solution $u \in W_2^1 (B_R)$ to equation \eqref{11}
solves also equation \eqref{16} if $\div b = 0$.
The converse statement is valid for bounded solutions.

\begin{theorem}[\cite{SSSZ}, Proposition 4.1]
\label{t5}
Let $u \in L_\infty (B_R)$,
$b\in L_2 (B_R)$, $\div b = 0$, and \eqref{16} be satisfied. 
Then $u \in W_2^1 (B_r)$ for all $r<R$, 
$u$ solves the equation \eqref{11}, and the estimate
$$
\| \n u \|_{L_2(B_r)} \le 
C (n, r, R) \left( 1 + \|b\|_{L_1(B_R)} \right)^{1/2} \| u \|_{L_\infty(B_R)} 
$$
holds.
\end{theorem}

In order to prove Theorem \ref{t4} we establish

\begin{theorem}
\label{t6}
Let $n \ge 3$, $p<n$.
There are two positive constants $c_0$, $c_1$ such that for any $\er>0$
there exist a vector-function $b_\er \in C^\infty (\overline{B_{1/2}})$,
$\div b_\er = 0$, $\|b_\er\|_{L_p(B_{1/2})} \le c_0$, and a scalar function
$u_\er \in C^\infty (\overline{B_{1/2}})$, $\|u_\er\|_{L_\infty(B_{1/2})} \le 1$,
$\|u_\er\|_{W_2^1(B_{1/2})} \le c_0$, which satisfy 
the equations \eqref{11} and \eqref{16}, and moreover
$$
u_\er (0) = 0, \quad u_\er (0, \dots, 0, 2\er) \ge c_1 .
$$
\end{theorem}

This result was proven in \cite{SSSZ} for $n=3$ and $p=1$.
It is also clear from the construction of $b_\er$ in \cite{SSSZ},
that one can take any power $p<2$.
However, in order to deduce Theorem \ref{t4} from Theorem \ref{t6}
one has to get Theorem \ref{t6} with a power $p \ge 2$.

In Section 2 we prove Theorem \ref{t12}. 
In Section 3 we prove Theorem \ref{t4} and Theorem \ref{t6}.
Some comments are collected in Section 4.

We do not consider the parabolic equation 
$\partial_t u - \D u + b \cdot \n u = 0$,
and the regularity of a solution in dependence of the properties
of a coefficient $b$.
Some results in this direction (under the condition $\div b = 0$)
can be found in \cite{NU, SSSZ, SVZ} (see also references therein).

The author is grateful to prof. G.~Ser\" egin for attracting his
attention to the problem.
Author thanks also A.~Nazarov, 
A.~Pushnitski and T.~Shilkin for valuable comments.

\subsection{Table 1: the local properties of a solution $u$ 
to equation \eqref{11} with $b \in L_p$}
$$
\begin{array}{|c|c|c|}
\hline
 & n=2 & n \ge 3 \\
\hline
& &\\
p>n & u \in C^{1, 1-n/p} & u \in C^{1, 1-n/p}\\
\hline
& &\\
 & \text{In general}\quad u \notin L_\infty, & \\
p=n & \text{or}\quad u \in L_\infty,\ u \notin C^\al . 
 & u \in C^\al \quad \forall \al<1\\
 & \text{If}\ \div b = 0 ,\ \text{then}\ 
u \in C^\al \ \forall \al<1 . & \\
\hline
& &\\
 & & \text{In general}\quad u \notin L_\infty. \\
p<n & --- & \text{It is also possible (even in the case} \\
& & \div b = 0 ) \ \text{that}\ u \in L_\infty,\ u \notin C . \\
\hline
\end{array}
$$

\section{Proof  of Theorem \ref{t12}}

\subsection{Existence of strong solution}
First, let us consider the Dirichlet problem 
for the Laplace equation in a ball
\begin{equation}
\label{n21}
- \D u = f \text{ in } B_R, \quad 
\left.u\right|_{\dd B_R} = 0 .
\end{equation}
Explicit formulas for the solution together with the Calderon-Zygmund 
estimates of singular integrals imply the well known

\begin{theorem}
\label{t24}
Let $f \in L_q (B_R)$, $1<q<\infty$.
There exists a unique function $u \in W_q^2 (B_R)$ satisfying \eqref{n21},
and $\| u \|_{W_q^2 (B_R)} \le C_1 \|f\|_{L_q (B_R)}$.
\end{theorem}

Now, let us consider the problem
\begin{equation}
\label{n22}
\begin{cases}
- \D v + b \cdot \n v = f \text{ in } B_R, \\
\left.v\right|_{\dd B_R} = 0 .
\end{cases}
\end{equation}
The following Lemma is also well known, 
we give a proof for the reader convenience.

\begin{lemma}
\label{l22}
Let $n \ge 2$, $1<q<n$.
There exists a positive number $\er_0 (n,q)$ 
such that if $b \in L_n (B_R)$, 
$\|b\|_{L_n (B_R)} \le \er_0$, $f \in L_q (B_R)$,
then there exists a unique function $v \in W_q^2 (B_R)$
satisfying \eqref{n22}.
Moreover, $\| v \|_{W_q^2 (B_R)} \le C \| f \|_{L_q (B_R)}$.
\end{lemma}

\begin{proof}
Denote by $L_0$ the Laplace operator of the Dirichlet problem,
$$
L_0 = - \D : W_q^2 \cap \mathring W_q^1 \to L_q .
$$
The operator $b \cdot \n L_0^{-1}$ is bounded in $L_q (B_R)$.
Indeed, let $f \in L_q (B_R)$,
$u = L_0^{-1} f \in W_q^2 (B_R)$.
Due to the imbedding theorem
$W_q^2 \subset W_{nq/(n-q)}^1$ we have
$$
\|b\cdot \n u\|_{L_q} \le \|b\|_{L_n} \|\n u\|_{L_{nq/(n-q)}}
\le C_0 \|b\|_{L_n} \|u\|_{W_q^2} \le C_0 C_1 \|b\|_{L_n} \| f \|_{L_q} ,
$$
where on the last step we used Theorem \ref{t24}.
If $\er_0 < (2 C_0 C_1)^{-1}$, then
$\|b \cdot \n L_0^{-1}\|_{L_q \to L_q} \le 1/2$.
Now, we set 
$$
v = L_0^{-1} \left(I + b \cdot \n L_0^{-1}\right)^{-1} f .
$$
Clearly,
$$- \D v + b \cdot \n v = f, \quad v \in W_q^2 \cap \mathring W_q^1, \quad
\text{and}\quad
\| v \|_{W_q^2 (B_R)} \le 2 C_1 \| f \|_{L_q (B_R)}, 
$$
as $\left\| \left(I + b \cdot \n L_0^{-1}\right)^{-1}\right\|_{L_q \to L_q} \le 2$.
\end{proof}

\subsection{Spaces $H_1$ and $BMO$}
Let us recall a definition of the Hardy space $H_1 (\R^n)$.
Let $\Phi \in C_0^\infty (B_1)$,
$\int_{B_1} \Phi(x) \, dx = 1$.
For $f \in L_1 (\R^n)$ we set
$$
\left(M_\Phi f\right) (x) = \sup_{t>0} 
\left|\frac1{t^n} \int_{\R^n} \Phi \left(\frac{x-y}t\right) f(y) \, dy \right| ,
$$
and
$$
H_1 (\R^n) = \{ f \in L_1 (\R^n) : M_\Phi f \in L_1 (\R^n) \}, 
\quad \| f \|_{H_1} = \|M_\Phi f\|_{L_1 (\R^n)} .
$$
The space $H_1$ does not depend on the choice of a function $\Phi$,
and the norms constructed with different functions $\Phi$ are equivalent.
A detailed exposition of the theory of Hardy spaces can be found in \cite{Stein}.
The dual space to $H_1$ is the space $BMO (\R^n)$ (Bounded Mean Oscillation).
Its definition read as follows:
a function $f \in L_{1, loc} (\R^n)$ belong to $BMO$ if and only if
$$
\sup_{x\in \R^n} \sup_{R>0} \frac1{\left|B_R(x)\right|}
\int_{B_R(x)} \left|f(y)- f_{B_R(x)}\right| \, dy 
=: \|f\|_{BMO} <\infty .
$$
Here
$f_{B_R(x)} = \frac1{\left|B_R(x)\right|} \int_{B_R(x)} f(y) \, dy$.
The functional $\|\, .\, \|_{BMO}$ is a seminorm
(it vanishes on the constants).
We will use the following result.

\begin{lemma}[\cite{CLMS}, Theorem II.1.2]
\label{l23}
Let $b\in L_p (\R^n)$, $1<p<\infty$, $\div b = 0$, $\ph \in W_{p'}^1 (\R^n)$.
Then $b\cdot \n\ph \in H_1 (\R^n)$,
$$
\|b\cdot \n\ph\|_{H_1} \le C \|b\|_{L_p} \|\n\ph\|_{L_{p'}} .
$$
\end{lemma}

Now, we can establish the following estimate.

\begin{lemma}
\label{l24}
Let $n=2$, $b \in L_2 (B_R)$, $\div b = 0$.
Then 
\begin{equation}
\label{23}
\left| \int_{B_R} b \cdot \n\ph \, \psi\, dx \right| \le 
C \|b\|_{L_2 (B_R)} \|\n\ph\|_{L_2 (B_R)} \|\n\psi\|_{L_2 (B_R)} 
\quad \forall \ \ph \in \mathring W_2^1 (B_R), \ \psi \in C_0^\infty (B_R) .
\end{equation}
\end{lemma}

\begin{proof}
First, as $\div b = 0$, we can represent the function $b$ as
$(b_1, b_2)= (\dd_2\oo, - \dd_1 \oo)$ with $\oo \in W_2^1 (B_R)$.
We extend the function $\oo$ into the whole plane,
and denote this extension by $\tilde \oo$,
$$
\tilde\oo \in W_2^1 (\R^2), \quad \left. \tilde\oo \right|_{B_R} = \oo, \quad
\|\tilde\oo\|_{W_2^1 (\R^2)} \le C \|\oo\|_{W_2^1 (B_R)} .
$$
Let us define a vector-function 
$\tilde b = (\dd_2\tilde\oo, - \dd_1\tilde\oo)$.
Clearly,
$$
\tilde b \in L_2 (\R^2), \quad \|\tilde b\|_{L_2(\R^2)} \le C \|b\|_{L_2(B_R)},
\quad \left.\tilde b\right|_{B_R} = b, \quad \div \tilde b = 0 .
$$
Therefore, by Lemma \ref{l23}, $\tilde b \cdot \n\ph \in H_1 (\R^2)$ and
$$
\|\tilde b \cdot \n\ph\|_{H_1} \le C \|b\|_{L_2 (B_R)} \|\n\ph\|_{L_2 (B_R)} .
$$
On the other hand, it is well known, that the space $W_2^1(\R^2)$ is imbedded
in $BMO (\R^2)$, and the estimate
$$
\|\psi\|_{BMO(\R^2)} \le C \|\n\psi\|_{L_2(\R^2)}
$$
holds (it is a simple consequence of the Poincar\'e inequality,
see for example \cite{Evans}).

Finally, the integral of a product of an $H_1$-function and a bounded 
$BMO$-function can be estimated by the product of the corresponding norms
(see \cite{Stein}),
$$
\left| \int_{\R^2} \tilde b \cdot \n\ph\, \psi\, dx \right| \le 
C \|\tilde b \cdot \n\ph\|_{H_1} \|\psi\|_{BMO} \le
C \|b\|_{L_2 (B_R)} \|\n\ph\|_{L_2 (B_R)} \|\n\psi\|_{L_2 (B_R)} .\quad
\qedhere
$$
\end{proof}

\begin{rem}
Lemma \ref{l24} is borrowed from the paper \cite{MV}.
In this paper a detailed investigation of the boundedness of the integral
in the left hand side of \eqref{23} under different conditions 
on $b, \ph, \psi$ is done.
We gave the proof of \eqref{23} in our particular case
for the convenience of a reader.
\end{rem}

\subsection{Uniqueness of weak solution}

\begin{lemma}
\label{l25}
Let $b\in L_2(B_R)$, $\div b = 0$.
Then the solution to the problem \eqref{n22}
is unique in the space $\mathring W_2^1 (B_R)$.
\end{lemma}

\begin{proof}
Let $u$ solve the homogeneous problem
\begin{equation}
\label{24}
- \D u + b \cdot \n u = 0, \quad u \in \mathring W_2^1 (B_R) .
\end{equation}
Choose a sequence $\psi_n \in C_0 ^\infty (B_R)$ such that
$\psi_n \to u$ in $W_2^1 (B_R)$.
Then
$$
\int_{B_R} |\n u|^2 dx \le \int_{B_R} \n u \cdot \n \psi_n dx + 
\|\n u\|_{L_2(B_R)} \|\n u - \n \psi_n\|_{L_2(B_R)} .
$$
The second term tends to 0 when $n \to \infty$. 
For the first term we have
$$
\int_{B_R} \n u \cdot \n \psi_n dx = 
- \int_{B_R} b \cdot \n u \psi_n dx = 
\int_{B_R} b \cdot \n (\psi_n - u) \psi_n dx ,
$$
where we used \eqref{24} and the equality 
$\int_{B_R} b \cdot \n \psi_n \psi_n dx = 0$ 
which is due to the divergence-free condition.
By virtue of Lemma \ref{l24},
$$
\left| \int_{B_R} b \cdot \n (\psi_n - u) \psi_n dx \right| \le
C \|b\|_{L_2(B_R)} \|\n \psi_n - \n u\|_{L_2(B_R)} \|\n \psi_n\|_{L_2(B_R)}
\underset{n\to \infty}{\longrightarrow} 0 .
$$
So,
$\|\n u \|_{L_2(B_R)}^2 = 0$, and $u \equiv 0$.
\end{proof}

\begin{rem}
Example 1 shows that the uniqueness of weak solution can be violated
in the case $\div b \neq 0$.
\end{rem}

\subsection{Proof of Theorem \ref{t12}.}
The statement of the Theorem is local.
Therefore, without loss of generality, we can assume that the norm
$\|b\|_{L_2(B_R)}$ is arbitrarily small.
Let $u \in W_2^1 (B_R)$ be a solution to the equation \eqref{11},
and let $\ze \in C_0^\infty (B_R)$, 
$\left.\ze\right|_{B_r} \equiv 1$.
Then
$$
- \D (\ze u) + b \cdot \n (\ze u) = 
- \D \ze u - 2 \n \ze \cdot \n u + 
b \cdot \n \ze u \in L_q (B_R) \quad \forall \ q <2 .
$$
Thus, the function $(\ze u)$ solves the problem \eqref{n22} with 
the right hand side in $L_q$.
By virtue of Lemma \ref{l22}, such a problem has a solution
from $W_q^2 (B_R)$.
On the other hand, the solution is unique due to Lemma \ref{l25}.
So, $u \in W_q^2 (B_r)$ for all $q<2$.
\qed

\begin{rem}
\label{r26}
Proof of Theorem \ref{t2} can be done similarly.
The existence of strong solution is due to Lemma \ref{l22}.
The uniqueness of weak solution is given by

\begin{lemma}
\label{l27}
Let $n \ge 3$.
There is a number $\er_1 = \er_1 (n)$ such that 
a solution to the problem \ref{n22} is unique in $\mathring W_2^1 (B_R)$ 
if $b \in L_n (B_R)$, $\|b\|_{L_n(B_R)} \le \er_1$.
\end{lemma}

\begin{proof}
Let $u$ be a solution to the problem \eqref{n22} with $f = 0$.
Using the H\" older inequality and the imbedding Theorem 
$W_2^1 \subset L_{2n/(n-2)}$ we have
$$
\int_{B_R} |\n u|^2 dx = - \int_{B_R} b \cdot \n u u \, dx 
\le \|b\|_{L_n(B_R)} \|\n u\|_{L_2(B_R)} \|u\|_{L_{\frac{2n}{n-2}}(B_R)} 
\le C_0 \|b\|_{L_n(B_R)} \|\n u\|_{L_2(B_R)}^2 .
$$
If $\er_1 < 1/C_0$, then $\|\n u\|_{L_2(B_R)} = 0$.
\end{proof}

Now, multiplying a solution to the equation \eqref{11} by a cut-off function,
we get the relation
$$
u \in W_q^1 (B_R), \ 2\le q < n \quad \Longrightarrow \quad 
u \in W_q^2 (B_r) \subset W_{\frac{nq}{n-q}}^1 (B_r), \quad \forall\ r < R .
$$
Iterating this relation $\left[\frac{n+1}2\right]$ times
we obtain $u\in W_q^2 (B_r)$ for all $q<n$ and $r<R$.
\end{rem}

\section{Proof of Theorem \ref{t4}}

The proof of Theorem \ref{t6} (with $p=1$) in \cite{SSSZ} 
is based on the theory of the stochastic processes. 
We prove Theorem \ref{t4} and Theorem \ref{t6}
following the general scheme of \cite{SSSZ},
but without using the probability theory.

\subsection{Coefficient $b$}
Let $n \ge 3$, let $\O$ be a cylinder in $\R^n$,
$$
\O = \{ x \in \R^n : \rho < 1, z \in (-1, 1) \},
$$
where
$\rho = \sqrt{x_1^2 + \dots + x_{n-1}^2}$, $z = x_n$.
We will use the auxiliary parameters
$\mu \in (1, 2)$, $\er \in (0, 1/2)$ and a function
$\eta \in C^\infty (\R)$, 
$\eta (t) = 0$ if $t \le 1/2$,
$\eta (t) = 1$ if $t \ge 1$.
Introduce the function
\begin{equation}
\label{20}
H_\er (x) = \rho^{n-1} z^{-\mu} \eta (z/\er) \eta (z/\rho)
\end{equation}
if $x_n \ge 0$, and 
$H_\er (x_1, \dots, x_{n-1}, x_n) =
- H_\er (x_1, \dots, x_{n-1}, -x_n)$ if $x_n < 0$.
It is clear that $H_\er \in C^\infty (\overline\O)$
if the dimension $n$ is odd, and 
$\rho^{-1} H_\er \in C^\infty (\overline\O)$
if $n$ is even.
We define the function $b_\er$ as follows
$$
b_\er (x) = K \rho^{1-n}
\left( x_1 \dd_z H_\er, \dots, x_{n-1} \dd_z H_\er, - \rho \dd_\rho H_\er \right) .
$$
In cylindrical coordinates it means that
\begin{equation}
\label{205}
(b_\er)_\rho = K \rho^{2-n} \dd_z H_\er, \quad
(b_\er)_z = - K \rho^{2-n} \dd_\rho H_\er,
\end{equation}
and all other components are zero.
Here $K$ is a large constant, which we choose later
(see Lemma \ref{l41} below);
it does not depend on $\er$.

\begin{lemma}
\label{l21}
The function $b_\er$ possesses the following properties:
\begin{itemize}
\item $b_\er \in C^\infty (\overline\O)$;
\item $\div b_\er = 0$;
\item we have
$$
(b_\er)_\rho = - \mu K \rho z^{-1-\mu}, \quad
(b_\er)_z = - (n-1) K z^{-\mu}
$$
on the set
\begin{equation}
\label{21}
\O_\er : = \{ x \in \O :\rho < z, \er < z <1 \}
\end{equation}
(it is a truncated cone in the upper half of the cylinder $\O$);
\item $b_\er \in L_p (\O)$ for $p < n/\mu$,
and the norms $\| b_\er \|_{L_p}$ are uniformly bounded with respect to $\er$.
\end{itemize}
\end{lemma}

\begin{proof}
The first three properties follows directly from the construction.
Let us verify the last one.
For postitve $z$ we have
$$
\left| \n H_\er (x) \right| \le C \rho^{n-1} z^{-\mu}
\left( \frac1\rho + \frac1 z + 
\frac1\er \chi_{[1/2, 1]} \left(\frac z\er\right) +
\frac z{\rho^2} \chi_{[1/2, 1]} \left(\frac z\rho\right) \right)
\chi_{[1/2, \infty)} \left(\frac z\rho\right) ,
$$
where $\chi_{[1/2, 1]}$ and $\chi_{[1/2, \infty)}$ are the characteristic
functions of the interval $[1/2, 1]$ and $[1/2, \infty)$ respectively.
Next,
$$
\frac1\er \chi_{[1/2, 1]} \left(\frac z\er\right)  \le \frac1 z, \quad
\frac z{\rho^2} \chi_{[1/2, 1]} \left(\frac z\rho\right) \le \frac1\rho,
\quad \text{and} \quad
\frac1 z \chi_{[1/2, \infty)} \left(\frac z\rho\right) 
\le \frac2\rho \chi_{[1/2, \infty)} \left(\frac z\rho\right)  .
$$
Therefore,
$$
\left| \n H_\er (x) \right| \le C \rho^{n-2} z^{-\mu}
\chi_{[1/2, \infty)} \left(\frac z\rho\right)
$$
and
\begin{equation}
\label{22}
\left| b_\er (x) \right| \le C K z^{-\mu}
\chi_{[1/2, \infty)} \left(\frac z\rho\right) ,
\end{equation}
where the constant $C$ depends on the function $\eta$ only
and does not depend on $\er$.
The last inequality implies
$$
\int_\O |b_\er(x)|^p dx \le 
C K^p \int_0^1 \rho^{n-2} d\rho \int_{\rho/2}^\infty z^{-\mu p} dz < \infty ,
$$
because $n - \mu p > 0$.
\end{proof}

\subsection{Auxiliary function $f$}

\begin{lemma}
\label{l31}
There exists a function $f \equiv f_\er \in C^2 [\er, 1]$
which possesses the following properties 

1) $f(z)\ge 0$, $f'(z)\ge 0$;

2)$f(\er) = 0$, $f(2\er) \ge c_1 >0$, $f(1) = 1$;

3) $f(z) \le c_2 f'(z) z^{2-\mu}$,
$- f''(z) \le c_3 f'(z) z^{-\mu}$.

Here the positive constants $c_1$, $c_2$, $c_3$ depend on $\mu$ 
and do not depend on $\er$.
\end{lemma}

\begin{rem}
Such a function can not exist when $\mu = 1$.
Indeed, the conditions
$$
f'(z) \ge 0, \quad f(2\er) \ge c_1 \quad \text{and} 
\quad f(z) \le c_2 f'(z) z
$$
imply that $f'(z) \ge c_1 c_2^{-1} z^{-1}$ when $z \ge 2\er$.
Therefore, 
$$
1-c_1 \ge f(1) - f(2\er) \ge c \int_{2\er}^1 \frac{dz}z = c |\ln 2 \er| ,
$$
and we have a contradiction.
\end{rem}

{\it Proof of Lemma \ref{l31}.}
First, we define the function
\begin{equation*}
h(t) = 
\begin{cases}
\frac12 \left(\er^{-3} - \er^{\mu-4}\right) t^2 -
\left(2 \er^{-2} - \er^{\mu-3}\right) t +
\left(2\er^{-1} + \frac2{2-\mu} \er^{\mu-2}\right), \ \ \er \le t \le 2\er, \\
\frac{2^{3-\mu}}{2-\mu} \, t^{\mu-2}, \ \ 2\er < t \le 1 .
\end{cases}
\end{equation*}
Its derivative
\begin{equation*}
h'(t) = 
\begin{cases}
\left(\er^{-3} - \er^{\mu-4}\right) t - 2 \er^{-2} + \er^{\mu-3},
\ \ \er \le t \le 2\er, \\
- 2^{3-\mu} t^{\mu - 3}, \ \ 2\er < t \le 1 ,
\end{cases}
\end{equation*}
is continuous and negative everywhere. 
Therefore, the function $h \in C^1 [\er, 1]$ is decreasing.

Put $g(z)= \int_\er^z h(t)\, dt$.
The function $g$ increases, $g \in C^2 [\er, 1]$ and $g(\er) = 0$.
We have
\begin{eqnarray*}
g (2\er) = \left(\er^{-3} - \er^{\mu-4}\right) \frac{7\er^3} 6 
- \left(2 \er^{-2} - \er^{\mu-3}\right) \frac{3\er^2} 2 
+ \left(2\er^{-1} + \frac2{2-\mu} \er^{\mu-2}\right) \er \\
= \frac16 + \left(\frac13 + \frac2{2-\mu}\right) \er^{\mu-1} > \frac16 ,
\end{eqnarray*}
and
\begin{eqnarray*}
g (1) = g (2\er) + \int_{2\er}^1 h(t) \, dt =
g (2\er) + \frac{2^{3-\mu}}{(2-\mu)(\mu-1)} \left(1 - (2\er)^{\mu-1}\right) \\
< \frac16 + \frac{2^{3-\mu}}{(2-\mu)(\mu-1)} + \frac{\er^{\mu-1}} 3
< \frac12 + \frac{2^{3-\mu}}{(2-\mu)(\mu-1)} =: d_\mu .
\end{eqnarray*}

Now, we define the function $f$ as $f(z) = g(z) / g(1)$.
It is immediate that the properties 1) and 2) are fulfilled;
one can take $c_1 = (6 d_\mu)^{-1}$.
Let us verify the property 3).
It is sufficient to check the corresponding inequalities
for the function $g$ instead of function $f$.
For $z\le 2\er$ we have
$$
g'(z) = h(z) \ge h(2\er) = \frac2{2-\mu} \er^{\mu-2}, \quad
g(z) \le g(1) < d_\mu \le C g'(z) z^{2-\mu},
$$
where 
$C = (2-\mu) d_\mu /2$.
Further,
$$
g'(z) z^{-\mu} \ge \frac2{2-\mu} \er^{\mu-2} (2\er)^{-\mu} 
= \frac{2^{1-\mu}}{2-\mu} \er^{-2}, \quad
g''(z)= h'(z)  \ge h'(\er) = - \er^2 ,
$$
therefore,
$$ 
- g''(z) \le (2-\mu) 2^{\mu-1} g'(z) z^{-\mu} .
$$

For $z > 2\er$ we have
$$
g'(z)= \frac{2^{3-\mu}}{2-\mu} z^{\mu-2} \quad \Longrightarrow \quad
g(z) \le C g'(z) z^{2-\mu},
$$ 
$C = (2-\mu) 2^{\mu-3} d_\mu$.
Finally,
$$
- g''(z) = (2-\mu) g'(z) \le (2-\mu) g'(z) z^{-\mu} . \quad \qed
$$

\subsection{Barrier function $v$}
\label{s4}
Let $f = f_\er$ be a function constructed in Lemma \ref{l31}.
Consider the function 
$v_\er (z) = f(z) \cos \frac{\pi\rho}{2z}$
on the set $\O_\er$ defined by \eqref{21}.
Clearly, $v_\er \in C^2 (\overline{\O_\er})$, 
\begin{equation}
\label{345}
v_\er \ge 0 \ \text{in} \ \O_\er, \quad
\left. v_\er\right|_{z=\er} = 0, \quad 
\left. v_\er\right|_{z=\rho} = 0, \quad 
\left. v_\er\right|_{z=1} = \cos \frac{\pi\rho}2 ,
\end{equation}
and
\begin{eqnarray*}
\dd_\rho v_\er = - \frac\pi{2z} f(z) \sin \frac{\pi\rho}{2z}, \quad
\dd_\rho^2 v_\er = - \frac{\pi^2}{4z^2} f(z) \cos \frac{\pi\rho}{2z}, \\
\dd_z v_\er = f'(z) \cos \frac{\pi\rho}{2z} +
\frac{\pi\rho}{2z^2} f(z) \sin \frac{\pi\rho}{2z}, \\
\dd_z^2 v_\er = f''(z) \cos \frac{\pi\rho}{2z} +
\frac{\pi\rho}{z^2} f'(z) \sin \frac{\pi\rho}{2z} -
\frac{\pi\rho}{z^3} f(z) \sin \frac{\pi\rho}{2z} -
\frac{\pi^2 \rho^2}{4z^4} f(z) \cos \frac{\pi\rho}{2z} .
\end{eqnarray*}

\begin{lemma}
\label{l41}
Let the function $b_\er$ be defined by formulas \eqref{20}, \eqref{205}
with 
$$
K > \max \left(\frac{4n}{n-\mu-1}, 
\pi^2 c_2 + c_3\right) ,
$$
where $c_2$ and $c_3$ are the constants from Lemma \ref{l31}.
Then the inequality
$$ 
\D v_\er (x) - b_\er (x) \cdot \n v_\er (x) > 0
$$
holds in $\O_\er$.
\end{lemma}

\begin{proof}
We have
\begin{eqnarray*}
\D v_\er =
\dd_\rho^2 v_\er  + \frac{n-2}\rho \dd_\rho v_\er + \dd_z^2 v_\er \\
= \left( - \frac{\pi^2}{4z^2} f(z) -
\frac{\pi^2 \rho^2}{4z^4} f(z) + f''(z) \right) \cos \frac{\pi\rho}{2z} \\
+ \left( - \frac{(n-2)\pi}{2\rho z} f(z) - \frac{\pi\rho}{z^3} f(z)
+ \frac{\pi\rho}{z^2} f'(z) \right) \sin \frac{\pi\rho}{2z} \\
\ge \left( - \frac{\pi^2}{2z^2} f(z) + f''(z) \right) \cos \frac{\pi\rho}{2z} 
 - \frac{n\pi}{2\rho z} f(z) \sin \frac{\pi\rho}{2z} ,
\end{eqnarray*}
where we used the inequalities $\rho \le z$ in $\O_\er$
and $f'(z) > 0$.

Next,
$$
- b_\er \cdot \n v_\er =
(n-1) K z^{-\mu} f'(z) \cos \frac{\pi\rho}{2z} +
\frac{n-\mu-1}2 K \pi \rho z^{-2-\mu} f(z) \sin \frac{\pi\rho}{2z} .
$$
Taking into account Lemma \ref{l31}, we get
\begin{eqnarray}
\label{41}
\D v_\er - b_\er \cdot \n v_\er \ge
\left( (n-1) K - \frac{\pi^2}2 c_2 - c_3 \right) 
z^{-\mu} f'(z) \cos \frac{\pi\rho}{2z} \\ 
+ \left( \frac{n-\mu-1}2 K \pi \rho z^{-2-\mu} - \frac{n\pi}{2\rho z} \right) 
f(z) \sin \frac{\pi\rho}{2z} .
\nonumber
\end{eqnarray}

If $0<\rho \le z/2$ then 
$\sin \frac{\pi\rho}{2z} \le \frac{\pi\rho}{2z}$ and
$\cos \frac{\pi\rho}{2z} \ge \frac1{\sqrt 2}$, therefore
$$
\frac{n\pi}{2\rho z} f(z) \sin \frac{\pi\rho}{2z} \le 
\frac{n\pi^2}{4z^2} f(z) \le 
\frac{n\pi^2}{4} c_2 f'(z) z^{-\mu} \le
\frac{n\pi^2 \sqrt 2}{4} c_2 f'(z) z^{-\mu} \cos \frac{\pi\rho}{2z} ,
$$
where we have used Lemma \ref{l31} again.
Thus, $\D v_\er (x) - b_\er (x) \cdot \n v_\er (x) > 0$
when $\rho \le z/2$ due to the fact that
$$
K > \pi^2 c_2 + c_3 \quad \Longrightarrow \quad
(n-1) K > \left(\frac12 + \frac{n\sqrt 2} 4\right) \pi^2 c_2 + c_3 .
$$

If $z/2 <\rho < z$ then 
$4 \rho^2 \ge z^2 \ge z^{1+\mu}$ and
$\frac{n\pi}{2\rho z} \le 2 n \pi \rho z^{-2-\mu}$.
Therefore, the last term in the right hand side of \eqref{41} 
is positive, as $(n-\mu-1) K > 4n$.
\end{proof}

\begin{rem}
This construction does not work for $n=2$,
because we have used the positiveness of the multiplier
$(n-\mu-1)$ in \eqref{41}, and $\mu > 1$.
\end{rem}

\subsection{Proof of Theorem \ref{t4} and Theorem \ref{t6}}
{\it Proof of Theorem \ref{t6}.}
Let the sets $\O$, $\O_\er$ and 
the function $b_\er$ be defined as before.
Then $b_\er \in C^\infty$, $\div b_\er = 0$ and
the norms $\|b_\er\|_{L_p (\O)}$ are uniformly bounded 
with respect to $\er$.
Let $u_\er \in W_2^1(\O)$ be the unique solution to the problem
$$
\begin{cases}
- \D u_\er + b_\er \cdot \n u_\er = 0 \quad \text{in} \ \O ,\\
\left.u_\er\right|_{z= \pm 1} = \pm \cos \frac{\pi\rho}2, \ \ 
\left.u_\er\right|_{\rho = 1} = 0 .
\end{cases}
$$
Evidently,
$\left.u_\er\right|_{\overline B_1} \in C^\infty (\overline B_1)$
and $\|u_\er\|_{L_\infty (B_1)} = 1$. 
The norms $\|u_\er\|_{W_2^1(B_{1/2})}$ are also uniformly bounded
due to the Theorem \ref{t5}.
Next, it is clear that the function $u_\er$ is odd,
$$
u_\er (x_1, \dots, x_{n-1}, - x_n) =
- u_\er (x_1, \dots, x_{n-1}, x_n).
$$
Therefore, 
$\left.u_\er\right|_{z= 0} = 0$.
By the maximum principle, $u_\er(x) \ge 0$ when $z \ge 0$.
This means that $u_\er(x) \ge v_\er(x)$ on the boundary $\dd\O_\er$,
where $v_\er$ is the barrier function constructed in Section \ref{s4}
(see \eqref{345}).
Using the maximum principle for the set $\O_\er$ and the Lemma \ref{l31}, 
we get
\begin{equation}
\label{51}
u_\er (0, \dots, 0, z) \ge v_\er (0, \dots, 0, z) = f_\er(z) \ge c_1
\quad \forall z \ge 2\er . \ \ \qed 
\end{equation}

{\it Proof of Theorem \ref{t4}.}
Without loss of generality we can assume $p> n/2$.

We deduce Theorem \ref{t4} from the Theorem \ref{t6}.
Roughly speaking, we repeat here the argument of \cite{SSSZ}.
Put
\begin{eqnarray*}
H_0 (x) = \rho^{n-1} z^{-\mu} \eta (z/\rho) \quad 
\text{when} \ \  x_n \ge 0 , \\
H_0 (x_1, \dots, x_{n-1}, x_n) =
- H_0 (x_1, \dots, x_{n-1}, -x_n) \quad 
\text{when} \ \  x_n < 0 .
\end{eqnarray*}
Let
$$
(b_0)_\rho = K \rho^{2-n} \dd_z H_0, \quad
(b_0)_z = - K \rho^{2-n} \dd_\rho H_0,
$$
and all other components be zero.
The constant $K$ here is defined in Lemma \ref{l41}.
It is evident that 
$b_\er \to b_0$ a.e. as $\er \to 0$, and
$\left| b_\er (x) \right| \le C K z^{-\mu}
\chi_{[1/2, \infty)} (z/\rho)$
due to \eqref{22}.
Therefore, the same estimate has place for the function $b_0$,
$b_0 \in L_p$, and  $b_\er \to b_0$ in $L_p$ for all $p < n/\mu$.
This yields also that $\div b_0 = 0$.

By virtue of the Theorem \ref{t1} and the inequality \eqref{22},
the functions $u_\er$ are uniformly bounded in $W^2_p (U)$, for all 
subdomains $U$ with smooth boundaries 
such that $\overline U \subset \O \setminus \{0\}$.
The imbedding $W^2_p (U) \subset C (\overline U)$ is compact, 
therefore, there is a subsequence $\{ u_{\er_k} \}$
which converges uniformly on $\overline U$.
Furthermore, Theorem \ref{t5} implies that
the sequence $\{ u_{\er_k} \}$ is uniformly bounded in $W_2^1 (B_{1/2})$.
Without loss of generality one can assume that $u_{\er_k}$
tends pointwise to a function $u_0$,
$$
u_{\er_k} (x) \to u_0 (x) \quad \forall x \neq 0 ,
$$
and $u_{\er_k} \to u_0$ weakly in $W_2^1 (B_{1/2})$.
Clearly,
$\|u_0\|_{L_\infty (B_{1/2})} \le 1$.

We have for any $h \in C_0^\infty (B_{1/2})$
$$
\int u_0 \left( \D h + b_0 \cdot \n h \right) \, dx = 
\lim_{k\to\infty} \int u_{\er_k} 
\left( \D h + b_{\er_k} \cdot \n h \right) \, dx = 0 .
$$
Thus, the equations \eqref{11} and \eqref{16} are fulfilled for $u_0$, $b_0$.

Finally, the function $u_0$ is odd,
$u_0 (x_1, \dots, x_{n-1}, - x_n) = - u_0 (x_1, \dots, x_{n-1}, x_n)$,
but 
$$
u_0 (0, \dots, 0, z) \ge c_1, \quad \forall z > 0, 
$$
due to \eqref{51}.
Therefore, the function $u_0$ is discontinuous at the origin.
\qed


\section{Comments and remarks}

\subsection{Case $n=1$}
We do not consider the one-dimensional case, because the equation
$- u''(x) + b(x) u'(x) = 0$
admits an explicit solution
$$
u(x) = C_1 \int_0^x \exp \left(\int_0^y b(t)\, dt\right)\, dy + C_2.
$$

\subsection{On Stampacchia's Theorem}
It is announced in \cite{S} that a solution to \eqref{12} 
under the conditions \eqref{13} and $b \in L_n$ must be bounded 
\cite[Theorem 4.1]{S}, and therefore, H\" older continuous 
\cite[Theorem 7.1]{S} for all $n \ge 2$.
These Theorems are proven in \cite{S} for $n\ge 3$.
However, for $n=2$, both statement are false,
see Examples 1 and 2 in \S 1.
The reason is that the imbedding Theorem 
$W_2^1 \subset L_{2n/(n-2)}$ used in \cite{S} 
has no place when $n=2$.

\subsection{Morrey space}
Let us recall the definition of Morrey's spaces:
$$
M_q^\al (\O) = \{ f \in L_q (\O) : \|f\|_{M_q^\al}
= \sup_{B_r(x) \subset \O} r^{-\al}  \|f\|_{L_q(B_r(x))} < \infty \} .
$$
The following result is proved in \cite{NU}.

\begin{theorem}
\label{t7}
Let $a$ satisfy \eqref{13}, $b \in M_q^{\frac nq - 1} (B_R)$, $n/2 <q < n$,
$\div b = 0$. 
Let $u \in W_2^1 (B_R)$ solve the equation \eqref{12}.
Then $u \in C^\al (B_R)$ with some $\al >0$.
\end{theorem}

The H\" older inequality implies that 
$L_p \subset M_q^{\frac nq - \frac np}$, $1\le q \le p$.
Therefore, Theorem \ref{t4} shows that the power $(n/q - 1)$ 
in Theorem \ref{t7} is sharp.

\subsection{Space $L_{2, \ln}$}
\label{p44}
The following result has place.

\begin{theorem}
\label{t3}
Let $n=2$.
Assume that the coefficient $b$ satisfies the condition
\begin{equation}
\label{14}
\int_{B_R} |b(x)|^2 \ln (1+|b(x)|^2)\, dx < \infty.
\end{equation}
Let $u \in W_2^1 (B_R)$ be a solution to \eqref{11}.
Then 
\begin{equation*}
u \in \bigcap_{q<2} W_q^2 (B_r) \subset \bigcap_{\al < 1} C^\al (B_r) \quad 
\forall r < R .
\end{equation*}
\end{theorem}

Denote by $L_{2, \ln} (B_R)$ the space of measurable functions $b$
(modulo functions vanishing on the set of full measure)
satisfying \eqref{14} (clearly, $L_{2, \ln} \subset L_2$).
It is the Orlicz space corresponding to the function $t^2 \ln (1+t^2)$.
The theory of Orlicz spaces can be found for example in \cite{KR}.
Recall some basic facts on such space.
The  quantity
$$
\| b \|_{L_{2, \ln} (B_R)} = \inf \left\{ k > 0 : 
\int_{B_R} \left|\frac{b(x)}k\right|^2 
\ln \left( 1 + \left|\frac{b(x)}k\right|^2 \right) dx \le 1 \right\}
$$
is well defined for $b \in L_{2, \ln} (B_R)$.
One can show that this functional is a norm,
and that 
$$
\| b \|_{L_{2, \ln} (B_r)} \to 0 \quad \text{as} \quad r \to 0 .
$$

\begin{lemma}
\label{lem22}
Let $n=2$, $R\le 1$, $b \in L_{2, \ln} (B_R)$, 
$\psi \in \mathring W_2^1 (B_R)$.
Then $b \psi \in L_2 (B_R)$ and
$$
\|b \psi\|_{L_2 (B_R)} \le C_0 \| b \|_{L_{2, \ln} (B_R)} \|\n\psi\|_{L_2 (B_R)} ,
$$
where $C_0$ is an absolute constant.
\end{lemma}

\begin{proof}
Follows form the fact (see, for example, \cite[Theorem 7.15]{GT})
that all functions from $\mathring W_2^1 (B_R)$ satisfy the estimate
\begin{equation*}
\int_{B_R} \exp 
\left( \frac{|\psi(x)|^2}{a_1^2 \|\psi\|_{W_2^1 (B_R)}^2} \right) dx 
\le a_2 |B_R|
\end{equation*}
with two constants $a_1$, $a_2$, 
and the elementary inequality
$$
\xi \eta \le \xi \ln \xi + e^\eta, \quad \xi, \eta > 0 . \quad \qedhere
$$
\end{proof}

Now, the proof of Theorem \ref{t3} is similar to the proof of Theorem \ref{t12}.
The uniqueness of weak solution (an analogue of Lemma \ref{l25})
follows from the estimate
$$
\left|\int_{B_R} b \cdot \n u u \, dx\right|  
\le \|b u\|_{L_2(B_R)} \|\n u\|_{L_2(B_R)}  
\le C_1 \|b\|_{L_{2, \ln} (B_R)} \|\n u\|_{L_2(B_R)}^2 
\quad \forall \ u \in \mathring W_2^1 (B_R) 
$$
if the norm $\| b \|_{L_{2, \ln} (B_R)}$ is sufficiently small.

We borrowed the condition \eqref{14} from \cite{NU}.
Under the conditions \eqref{13} and \eqref{14} it is proven in \cite{NU}
that any solution to the equation \eqref{12} is H\" older continuous
(see comments at the end of \S 2 in \cite{NU}).
Note that the condition \eqref{14} can not be changed 
by the finiteness of the integral
$\int_{B_R} |b(x)|^2 \left(\ln (1+|b(x)|^2)\right)^\ga dx$ with any $\ga <1$
(see Example 1).

\subsection{Maximum principle}
If the coefficient $b$ satisfies the conditions of 
Theorems \ref{t12}, \ref{t2} or \ref{t3}, then a solution $u$ to 
the equation \eqref{12} satisfies the maximum principle 
\cite[Corollary 2.2 and comments at the end of \S 2]{NU}.
Examples 2) and 4) in Section 1 show that the conditions imposed on $b$
again can not be weakened.

\subsection{Open questions}
The following questions remain open.
\begin{itemize}
\item
Let $n \ge 3$, $b \in L_p (B_R)$, $2\le p < n$, and $\div b = 0$.
Whether a solution $u \in W_2^1 (B_R)$ to equation \eqref{11} 
should be bounded in $B_r$, $r<R$ ?
\item 
Let $n=2$, $b \in L_2 (B_R)$. 
Whether a solution $u \in W_2^1 (B_R) \cap L_\infty (B_R)$ 
to equation \eqref{11} should be continuous?
\end{itemize}


\begin{thebibliography}{}

\bibitem{CLMS}
R.~Coifman, P.~L.~Lions, Y.~Meyer, S.~Semmes, 
{\it Compensated compactness and Hardy spaces},
J. Math. Pures Appl. 72 (3) (1993), p. 247–-286.

\bibitem{Evans}
L.~C.~Evans, {\it Partial Differential Equations},
AMS, Providence, Rhode Island, 1998.

\bibitem{GT}
D.~Gilbarg, N.~S.~Trudinger, {\it Elliptic Partial Differential
Equations of Second Order}, Springer-Verlag, 1983.

\bibitem{KR}
M.~A.~Krasnosel'skii, Ya.~B.~Rutickii, 
{\it Convex functions and Orlicz spaces}, 
P. Noordfoff Ltd., Groningen, 
1961.

\bibitem{LU}
O.~A.~Ladyzhenskaya, N.~N.~Ural'tseva, {\it Linear and Quasilinear
Equations of Elliptic Type}, Moscow, "Nauka", 1973.

\bibitem{MV}
V.~G.~Mazja, I.~E.~Verbitsky, {\it Form boundedness of 
the general second-order differential operator},
Comm. Pure Appl. Math. LIX (2006), p. 1286-–1329.

\bibitem{NU}
A.~I.~Nazarov, N.~N.~Ural'tseva, {\it The Harnack inequality and
related properties for solutions to elliptic and parabolic equations 
with divergence-free lower-order coefficients}, 
Algebra i Analiz, 23, 1 (2011), p. 136--168 (Russian);
Engl. transl. in St.Petersburg Math. Journal, 23 (2012), p. 93--115.

\bibitem{SSSZ}
G.~Seregin, L.~Silvestre, V.~ \v Sver\' ak, A.~Zlato\v s,
{\it On divergence-free drifts}, Journal of Differential Equations,
252 (2012), p. 505--540.

\bibitem{SVZ}
L.~Silvestre, V.~Vicol, A.~Zlato\v s,
{\it On the loss of continuity for super-critical drift-diffusion equations},
to appear in Arch. Rat. Mech. Anal.

\bibitem{S}
G.~Stampacchia, {\it Le probl\` eme de Dirichlet pour les \' equations
elliptiques du second ordre \` a coefficients discontinus},
Ann. Inst. Fourier (Grenoble), 15, 1 (1965), p. 189--258.

\bibitem{Stein}
E.~M.~Stein, {\it Harmonic Analysis},
Princeton University Press, Princeton, NJ, 1993.
\end{thebibliography}
\end{document}